\documentclass[11pt]{article}

\usepackage{amsfonts}
\usepackage{amssymb}
\usepackage{amsmath}
\usepackage{amsthm}
\usepackage{graphicx}
\usepackage{color}

\newtheorem{notation}{Notations}[section]
\newtheorem{remarque}[notation]{Remark}
\newtheorem{thm}[notation]{Theorem}

\newtheorem{prop}[notation]{Proposition}

\newcommand{\gga}{\gamma}            

\newcommand{\gphi}{\varphi}

\newcommand{\gP}{\Phi}
\newcommand{\gD}{\Delta}

\newcommand{\gl}{\lambda}

\newcommand{\eps}{\varepsilon}

\newcommand{\R}{\mathbb R}

\newcommand{\N}{\mathbb{N}}

\newcommand{\cC}{\mathcal C}
\newcommand{\cH}{\mathcal{H}}
\newcommand{\bbar}{\big{|}}

\newcommand{\hausmoins}{\mathcal{H}^{m-1}}

\newcommand{\dive}{\operatorname{div}}

\newcommand{\sign}{\operatorname{sign}}

\newcommand{\totvarn}[1]{\int_{\R^m} |D_\gga #1| }

\newcommand{\totepsn}[2]{\int_{#2} \sqrt{\eps^2+ |D_\gga #1|^2} d \gga}
\newcommand{\Ldeufin}{L^2_\gga (\R^m)}

\newcommand{\Pn}{\tilde{P}}
\newcommand{\tE}{\tilde{E}}

\begin{document}
\title{\Large A geometric approach for convexity in some variational problem in the Gauss space }

\author{
        M. Goldman
        \footnote{CMAP, CNRS UMR 7641, Ecole Polytechnique,
        91128 Palaiseau, France, email: goldman@cmap.polytechnique.fr}
       }
\date{}
\maketitle

\begin{abstract}
\noindent In this short note we prove the convexity of minimizers of some variational problem in the Gauss space. This proof is based on a geometric version of an older argument due to Korevaar.
\end{abstract}

\section{Introduction}
In the paper \cite{CGN} we prove together with A. Chambolle and M. Novaga, the convexity of the minimizers of the following variational problem in the Wiener space $X$:
\begin{equation}\label{probgenF}
 \min \int_X F(\nabla u)  +\frac{(u-g)^2}{2} d\gga
\end{equation}
under the general hypothesis that $F$ and $g$ are convex functions. The idea is to approximate the infinite dimensional problem by a finite dimensional one, to show convexity of the minimizers of the finite dimensional problems and prove convergence of these minimizers towards the minimum of \eqref{probgenF}.  In \cite{CGN}, we followed the approach of Alvarez Lasry and Lions \cite{ALL} to prove convexity in finite dimension. The aim of this note is to show an alternative proof based on ideas of Korevaar \cite{kore} when $F$ is the total variation (which was our main motivation in \cite{CGN}). More precisely, we will show that for $g \in \Ldeufin$ a convex function then the solution of
 
 \begin{equation}\label{princprobu}
 \min_{ BV_\gga(\R^m) \cap \Ldeufin} \totvarn{u} +\frac{1}{2} \int_{\R^m} |u-g|^2 d\gga
 \end{equation}
 is convex. As a by-product of our analysis we will also 
get that the minimizers of the Ornstein-Uhhlenbeck functional 
\[\min_{ H^1_\gga(\R^m) \cap \Ldeufin} \int_{\R^m} \frac{|\nabla u|^2}{2} +\frac{1}{2} \int_{\R^m} |u-g|^2 d\gga\]
are convex if $g$ is convex.

\noindent The plan of the note is the following. 
In Section 2 we recall some notation about functions of bounded variation and in Section 3 we show the convexity of the minima of \eqref{princprobu}.

\medskip 
\noindent{\bf Acknowledgements.} 
I warmly thank Matteo Novaga for the numerous discussions we had on this problem and for his constant support. I would also like to acknowledge the hospitality of the Scuola Normale Superiore di Pisa where this work has been done.
\section{Notation and preliminary results}

\noindent Let $m \in \N$ be fixed and let $\gga$ be the standard Gaussian measure on $\R^m$. Let us denote by $\Ldeufin:=L^2(\R^m,\gga)$.

\noindent We now give the definitions of  Sobolev spaces and functions of bounded variation in the Gauss space. 
For a smooth function $\Phi: \R^m \to \R^m$,   we define 
\[\dive_\gga \Phi (x):= \dive \Phi -\Phi\cdot x. \]

\noindent The operator $\dive_\gga$ is the adjoint of the gradient so that for every $u \in \cC^1_c(\R^m)$ and every $\Phi \in \cC^1_c(\R^m,\R^m)$, the following integration by parts holds:
\begin{equation}\label{integpart}
\int_{\R^m} u \dive_\gga \Phi \,d\gga =-\int_{\R^m} \nabla u\cdot \Phi d\gga.
\end{equation}
\noindent We will denote by $H^1_\gga(\R^m)$ the closure of the gradient in $\Ldeufin$. From this, formula \eqref{integpart} still holds for $u \in H^1_\gga(\R^m)$ and $\Phi \in \cC^1_c(\R^m,\R^m)$.

\noindent Given $u\in L^1_\gga(\R^m)$ we say that $u\in BV_\gga(\R^m)$ if
\[\totvarn{u}=\sup \left\{ \int_{\R^m} u \dive_\gga \Phi \, d\gga; \; \Phi \in \cC_b^1(\R^m,\R^m), \; |\Phi(x)|\le 1 \; \forall x\in \R^m\right\}<+\infty.\]  \\

\noindent We see that functions in $BV_\gga(\R^m)$ are in $BV_{\textrm{loc}}(\R^m)$ and that $D_\gga u= \gga Du$ so that most of the properties of classical $BV$ functions extends to the function in $BV_\gga(\R^m)$ (see \cite{AFP}). In particular for every set $E$ of finite Gaussian perimeter, the reduced boundary $\partial^* E$ of $E$ is rectifiable and every point of this reduced boundary has an outward normal $\nu^E$. Defining the sets $E^s$ by
\[E^{(s)}= \left\{ x\in \R^m \; :\; \lim_{r\to 0} \frac{ |E\cap B_r(x)|}{|B_R|}=s\right\}\]
then we have $\cH^{m-1}((E^{(0)}\cup E^{(1)}\cup \partial^*E)^c)=0$ for every set of finite Gaussian perimeter. Here we denoted by $\cH^{m-1}$ the $m-1$ dimensional Hausdorff measure.

\noindent We finally recall some facts about pairings between measures and bounded functions (see \cite{anzel} for more details). \\
\noindent We define the space $X_2$ to be the space of bounded functions $z$ with $\dive_\gga z \in \Ldeufin$.  For every smooth open set $\Omega$, the trace $[z\cdot\nu]$ can be defined in such a way that the integration by part formula
\[\int_{\Omega} (z,D_\gga u) d\gga +\int_{\Omega} u \dive_\gga z d\gga=\int_{\partial \Omega} [z\cdot\nu] u \gga(x) \cH^{m-1}\]
holds for $z\in X_2$ and $u\in BV_\gga \cap \Ldeufin$ where as usual $(z, D_\gga u)$ is the measure defined by
\[\int_{\R^m} (z\cdot D_\gga u) d\gga= -\int_{\R^m} u \varphi \dive_\gga z d\gga -\int_{\R^m} u z\cdot \nabla \varphi d\gga\]
for every $\varphi \in \cC^1_c(\R^m,\R^m)$.

\section{Convexity of the minimizer}

In this section we are going to prove the following result:

\noindent Let $g \in \Ldeufin$ be a convex function then the minimizer of
\begin{equation}\label{probun}
\min_{ BV_\gga \cap \Ldeufin} \totvarn{u} +\frac{1}{2} \int_{\R^n} |u-g|^2 d\gga
\end{equation}
is a convex function.

\noindent As in many other papers involving the total variation, we are going to study first the regularized problem:
\begin{equation}\label{probepsn}
\min_{ BV_\gga \cap \Ldeufin} J_\eps(u)=\totepsn{u}{\R^m} +\frac{1}{2} \int_{\R^m} |u-g|^2 d\gga
 \end{equation}
where as usual, if the Radon-Nikodym decomposition of $D_\gga u$ is given by $D_\gga u = \nabla u d\gga + D^s_\gga u$  we let
\[\totepsn{u}{\R^m}=\int_{\R^m} \sqrt{\eps^2 +|\nabla u|^2}d\gga + |D^s_\gga u|(\R^m).\]
As a simple consequence of the Reshetnyak's continuity Theorem we have that $J_\eps$ is lower semicontinuous for the $L^2_\gga(\R^m)$ convergence (see  \cite{AFP}).

\noindent We start by studying the Dirichlet problem  on  balls, namely

\begin{equation}\label{probepsnball}
\min_{ BV_\gga(B_R)} \totepsn{u}{B_R} +\frac{1}{2} \int_{B_R} |u-g|^2 d\gga + \int_{\partial B_R} |u-M| \gga(x) d\hausmoins(x).
 \end{equation}

\noindent Here $B_R$ is the ball of radius $R$ centered in the origin and $M$ is a constant to be chosen later. The term $\int_{\partial B_R} |u-M| \gga(x) d\hausmoins(x)$ can be seen as a Dirichlet term (see \cite{giaquinta} and \cite{andreu}). In the following we will note by $F(p)=\sqrt{\eps^2 +|p|^2}$.

\noindent On bounded domains, by Theorem 6.7 in \cite{andreu} we can give a characterization of the minimizers of \eqref{probepsnball}
\begin{thm}[Characterization of the minima]\label{subdif}
 A function $u\in BV_\gga (B_R)$ minimizes \eqref{probepsnball} if and only if $\frac{\nabla u}{\sqrt{\eps^2+|\nabla u|^2}} \in X_2$ and
 \begin{align*}
 &-\dive_\gga \left(\frac{\nabla u}{\sqrt{\eps^2+|\nabla u|^2}}\right) +u=g, \qquad  \qquad \frac{\nabla u}{\sqrt{\eps^2+|\nabla u|^2}}\cdot D^s_\gga u=|D^s_\gga u| \quad |D^s_\gga u|-a.e. \\[10pt]
  &\textrm{ and } \qquad  [\frac{\nabla u}{\sqrt{\eps^2+|\nabla u|^2}}\cdot \nu] \in \sign (M-u) \quad  \cH^{m-1}-a.e. \textrm{ in } \partial B_R.\end{align*}
where $\nu$ is the outward normal to $B_R$.
\end{thm}

 \noindent Adapting  very slightly the proof of \cite[Th.~5.16]{andreu}, we get the following comparison principle:
\begin{prop}[Comparison]\label{comparison}
Let $g_1 \ge g_2$ and $\gphi_1 \ge \gphi_2$ then the minimizers $u_i$ with $i=1,2$ of

\[
\min_{ BV_\gga(B_R)} \int_{B_R} F(D_\gga u) d \gga +\frac{1}{2} \int_{B_R} |u-g_i|^2 d\gga + \int_{\partial B_R} |u-\gphi_i| \gga(x) d\hausmoins(x)
 \]

\noindent verify $u_1 \ge u_2$.
\end{prop}

With this comparison property in hands, we can prove that for $M$ large enough, the minimizer of \eqref{probepsnball} makes vertical contact angle with the boundary of $B_R$. In the following, we will say that a function $v$ is a supersolution of \eqref{probepsnball} if it minimizes the functional with $\tilde{g}\ge g$ and $\varphi \ge M$.

\begin{prop}[vertical contact angle]\label{contact}
If  $C \ge \frac{m}{\eps r}+ \frac{R}{\eps}+r+|g|_{L^\infty(B_R)}$, then 
\[v(x)=\begin{cases} C-\sqrt{r^2-(x-x_0)^2} \quad \textrm{if } x \in B_r(x_0) \\[8pt]
M \qquad \textrm{otherwise}
\end{cases}\]
is a supersolution of \eqref{probepsnball}  if $B_r(x_0) \subset B_R$. Then for $M>C$, the minimizer of \eqref{probepsnball} has vertical contact angle with $\partial B_R$.
\end{prop}

\begin{proof}
We must show that for $C$ large enough,
\[-\dive_\gga (\nabla F (\nabla v))+v-g \ge 0.\]

\noindent A direct computation shows that in $B_r(x_0)$ we have $\displaystyle \nabla v= \frac{x-x_0}{\sqrt{r^2-(x-x_0)^2}}$ thus
\[\nabla F(\nabla v)= \frac{x-x_0}{\sqrt{\eps^2 r^2 +(1-\eps^2) |x-x_0|^2}}.\]
From this we get that
\begin{align*}
-\dive_\gga (\nabla F (\nabla v))+v-g \ge & -\frac{m}{\eps r} +\frac{x-x_0}{\sqrt{\eps^2 r^2 +(1-\eps^2) |x-x_0|^2}}\cdot x +C
\\
& -\sqrt{r^2-(x-x_0)^2} -|g|_{L^\infty(B_R)}
\\
\ge & -\frac{m}{\eps r}-\frac{|x-x_0|}{\sqrt{\eps^2 r^2 +(1-\eps^2) |x-x_0|^2}} |x| +C\\
&-r-|g|_{L^\infty(B_R)}\\
\ge &  -\frac{m}{\eps r}- \frac{R}{\eps} +C-r-|g|_{L^\infty(B_R)}
\end{align*}
Thus if $C\ge \frac{m}{\eps r}+ \frac{R}{\eps}+r+|g|_{L^\infty(B_R)}$ then $v$ is a super-solution.

\noindent If $M>C$, then considering balls of radius $r$ such that $\partial B_r \cap \partial B_R$ is reduced to a point, by the comparison Theorem \ref{comparison}, if $u$ minimizes \eqref{probepsnball} then $M>C\ge v\ge u$ and thus by Theorem \ref{subdif} we have 
\[[\frac{\nabla u}{\sqrt{\eps^2+|\nabla u|^2}} \cdot \nu] =1 \qquad \hausmoins- a.e. \textrm{ on } \partial B_R\]
which is the vertical contact angle condition.
\end{proof}

\noindent The interior regularity of minimizers of \eqref{probepsnball} easily follows by a result of Giaquinta, Modica and Soucek \cite{giaquinta}.

\begin{prop}\label{regularity}
Let $g$ be a $\cC^\alpha$ function then the minimizer of \eqref{probepsnball} is $\cC^{2,\alpha}(B_R)$.
\end{prop}

\begin{proof}
By Theorem 3.3 of \cite{giaquinta} we have that minimizers of
\[\min_{ BV_\gga(B_R)} \int_{B_R} F(Du) d\gga + \int_{B_R} G(x,u) d\gga + \int_{\partial B_R} |u-M| \gga(x) d\hausmoins(x)\]
are locally Lipschitz if $G(x,u)$ verifies the following hypothesis:
\begin{itemize}
\item $\bbar\displaystyle \frac{\partial G}{\partial u}\bbar+\bbar\frac{\partial^2 G}{\partial u \partial x}\bbar\le C$.
\item $\displaystyle \frac{\partial^2G}{\partial u^2} \ge 0$.
\end{itemize}
 Originally we have $G(x,u)=\frac{1}{2} |u-g(x)|^2$ which does not verifies exactly the hypothesis. However if we set $\tilde{G}(x,u)=\Psi(u)-g(x)u+\frac{1}{2} g(x)^2$ where $\Psi(u)=\frac{1}{2}u^2$ if $u\le C$ and $\Psi$ convex, $\cC^2$ with linear growth at infinity then $\tilde{G}$ verifies the condition mentioned above. The Euler-Lagrange equation verified by the minimizers with $\tilde{G}$ instead of $G$ is 
 \begin{equation}\label{tilde}\frac{\partial \Psi}{\partial u} +\partial \gP_\gphi(u)=g(x).\end{equation}
 Now we can apply Theorem 3.3 of \cite{giaquinta} to find that solutions of \eqref{tilde} are locally Lipschitz. Exactly as in Proposition \ref{comparison} the comparison principle holds for this equation and thus $M$ (respectively $-M$) is a supersolution (respectively a subsolution). This implies that if $C\ge M$ solutions of \eqref{tilde} are also solutions of \eqref{probepsnball} which are thus locally Lipschitz. By classical regularity theory for elliptic equations (see \cite{GT}) this implies that the solutions are indeed $\cC^{2,\alpha}(B_R)$.
\end{proof}

\begin{remarque}\rm
This proposition in particular applies for $g$ convex since convex functions are locally Lipschitz.

\end{remarque}
\noindent Having only interior regularity it is not possible to directly apply the results of Korevaar \cite{kore} which need continuity up to the boundary. The idea will be to use a geometric version of Korevaar's argument to get the convexity of the minimizers. 

\noindent For simplicity, in this part of the proof we focus on the case $\eps=1$.  By rescaling, the general case of $\eps \neq 1$ can be easily recovered (the Gaussian measure $\gga$ is not invariant by this scaling but it does not matter). Consider now the set (see Figure \ref{setE})
\begin{equation}\label{E}
E=\left\{(x,t)\in B_R \times [-M;M] \, / \, t<u(x)   \right\}. 
\end{equation}

\begin{figure}[ht]
\centering{\input{tildeE.pstex_t}}

     
\caption{}
\label{setE}
\end{figure}

\noindent The aim is to show that $E$ is a concave set. First we need to show that $E$ is regular. For this we follow an idea of Giusti (see \cite{giusti} and \cite{giusti2}) showing that $E$ is a solution of a certain obstacle problem.

\noindent For $F$ a set of finite perimeter in $\R^{m+1}$ let $\Pn(F)$ be defined by
\[\Pn(F)=\int_{\partial^* F} \gga(x) d\cH^m(x,t). \]

\noindent $\Pn$ is thus the perimeter associated to the measure $\mu(x,t)=\gga(x) dx dt$. Let now $H(x,t)=(t-g(x))\gga(x)$ then we have the following:

\begin{prop}\label{regE}
The set $\tE=E\cup (B_R^c\times [-M;M])$, where $E$ is defined in \eqref{E}, is a minimizer of
\begin{equation}\label{obstacle}
\Pn(F)+\int_F H(x,t) \;  dx\, dt
\end{equation}
among all sets containing $B_R^c \times [-M;M]$. As a consequence $\partial \tE$ is $\cC^1$.
\end{prop}


\begin{proof}
 Let us define the field

\[z(x,t)=\left\{\begin{array}{ll}\displaystyle \left(-\frac{\nabla u}{\sqrt{1+|\nabla u|^2}},\frac{1}{\sqrt{1+|\nabla u|^2}}\right)& \quad (x,t) \in B_R \times ]-M;M[\\[16pt]
 -\nu^{B_R}(x) & \quad (x,t) \in \partial B_R \times ]-M;M[\end{array}\right. .\]

\noindent Then $z$ is a $X_2$ vector field in $B_R\times ]-M;M[$ satisfying $|z|_{\R^{m+1}}=1$ and $[z\cdot \nu^{\tE}]=1$ where $\nu^{\tE}$ is the outward normal to $\tE$. Moreover if $z=(z',z_{m+1})$ with $z'\in \R^m$ and $z_{m+1}\in \R$ then setting by a slight abuse of notations
\[\dive_\gga z= \dive_\gga z' + \frac{\partial z_{m+1}}{\partial t}\]
 we have $\dive_\gga z= g-u$. Hence if $F\gD \tE\subset B_R \times ]-M;M[$, as $t<u(x)$ in $E$,


\begin{align*}
\int_{\tE\backslash F} (\dive_\gga z) d\mu=\int_{\tE\backslash F} (g-u)d\mu & \le \int_{\tE\backslash F} (g(x)-t) d\mu\\
&=-\int_{\tE\backslash F}H(t,x) dx dt= \int_{\tE\cap F} H dx dt-\int_{\tE} H dx dt.
\end{align*}
\noindent On the other hand we have:

\[\int_{\tE\backslash F} (\dive_\gga z) d\mu=\int_{\partial^*(\tE\backslash F)} [z\cdot \nu^{\tE\backslash F}] \gga(x) d\cH^m(x,t)\]
But $\tE\backslash F= \tE\cap F^c$ and as noticed by Figalli, Maggi and Pratelli in \cite{FMP},
\[\partial^*(\tE\cap F^c)=J_{\tE,F^c} \cup \left(\partial^* \tE\cap (F^c)^{(1)}\right) \cup \left( \partial^* F^c \cap {\tE}^{(1)} \right)\]
where $J_{\tE,F^c}=\left\{x\in \partial^*\tE\cap \partial^* F^c / \nu^{\tE}=\nu^{F^c}\right\}$. Moreover we have:
\[\nu^{\tE\backslash F}=\left\{\begin{array}{ll}
\nu^{\tE} & \textrm{in } \partial^* \tE\cap (F^c)^{(1)}\\
\nu^{F^c}=-\nu^F & \textrm{in } \partial^* F^c \cap {\tE}^{(1)}\\
\nu^{\tE}=-\nu^F & \textrm{in } J_{\tE,F^c}
\end{array}\right. .\]

\noindent From this we find

\begin{align*}
\int_{\tE\backslash F} (\dive_\gga z) d\mu&= \int_{\partial^* \tE \cap F^{(0)}} \gga d\cH^m -\int_{\partial^*F\cap {\tE}^{(1)}} \nu^F \cdot z \gga d\cH^m + \int_{J_{\tE,F^c}} [z\cdot \nu^{\tE}] \gga d\cH^m\\
&\ge \int_{\partial^* \tE\cap F^{(0)}} \gga d\cH^m-\int_{\partial^*F\cap {\tE}^{(1)}} \gga d\cH^m + \int_{J_{\tE,F^c}} [z\cdot \nu^{\tE}] \gga d\cH^m.
\end{align*}
We thus find:
\[\int_{\tE\cap F} H dx dt-\int_{\tE} H dx dt\ge \int_{\partial^* \tE\cap F^{(0)}} \gga d\cH^m-\int_{\partial^*F\cap \tE^{(1)}} \gga d\cH^m + \int_{J_{\tE,F^c}} [z\cdot \nu^{\tE}] \gga d\cH^m.\]

\noindent Similarly, studying what happens on $F\backslash \tE$ we get:
\[\int_{\tE\cap F} H dx dt-\int_F H dx dt\le \int_{\partial^* F\cap {\tE}^{(0)}} \gga d\cH^m-\int_{\partial^*\tE\cap F^{(1)}} \gga d\cH^m + \int_{J_{F,{\tE}^c}} [z\cdot \nu^F] \gga d\cH^m.\]

\noindent Summing these two inequalities and using that $\int_{J_{\tE,F^c}} [z\cdot \nu^{\tE}] \gga d\cH^m=\int_{J_{F,{\tE}^c}} [z\cdot \nu^F] \gga d\cH^m$ we have:
\begin{eqnarray*}
&&\int_{\partial^*F\cap({\tE}^{(0)}\cup {\tE}^{(1)})} \gga(x) d\cH^m(x,t)+\int_F H(x,t) dx dt \ge
\\
&& \int_{\partial^*\tE\cap(F^{(0)}\cup F^{(1)})} \gga(x) d\cH^m(x,t)+\int_{\tE} H(x,t) dx dt. 
\end{eqnarray*}
Adding to this equality $\int_{\partial^*\tE \cap \partial^*F} \gga(x) d\cH^m(x,t)$ 
and using that $\cH^m((A^{(1)}\cup A^{(0)} \cup \partial^*A)^c)=0$ 
for every set  of finite perimeter $A\subset \R^{m+1}$, we find as desired that
\[
\int_{\partial^*F} \gga(x) d\cH^m(x,t)+\int_F H(x,t) dx dt \ge\int_{\partial^*\tE} \gga(x) d\cH^m(x,t)+\int_{\tE} H(x,t) dx dt. 
\]

\noindent The regularity of $\partial \tE$ follows from an old paper of Miranda \cite{miranda}. We point out that in the paper cited above, the results are written for the classical perimeter without curvature terms. However, the argument is based on a blow-up procedure under which our functional reduces to the classical perimeter.
\end{proof}

We can now prove the concavity of $\tE$.
\begin{prop}
The set $\tE$ is concave thus $u$ is convex. 
\end{prop}
\begin{proof}
We will show that the set $U= \overline{E^c}$ is convex (see Figure \ref{setE}). Let us define for every $z=(x,t) \in \overline{B}_R\times[-M;M]$ the vertical distance of $z$ to $U$ by
\[d^v(z,U)=\inf\left(|t-t'| \, /\,  (x',t') \in U\right).\]

\noindent The function $d^v$ is continuous since $\partial U$ is a $\cC^1$ surface by Proposition \ref{regE}. $U$ is a compact set thus the function 
\[C(\gl, z,z')=d^v(\gl z +(1-\gl) z', U) \qquad \textrm{for }(\gl,z , z') \in [0;1]\times U\times U\]
attains its maximum. If this maximum is zero then $U$ is convex and we are done. Assume on the contrary that this maximum is positive. 

\noindent By the vertical contact angle condition we can assume that this maximum is attained at points $z$ and $z'$ in the interior of $\overline{B}_R\times[-M;M]$. Moreover, if $z=(x,t)\in U$, by decreasing $t$ (which increases $C$), we can assume that $t=u(x)$. Analogously we can assume that $z'=(x',u(x'))$. Then we find
\[C(\gl, z,z')=u(\gl x+(1-\gl) x')-\gl u(x) -(1-\gl)u(x').\]

\noindent We are thus in the situation of applying Korevaar's concavity maximum principle \cite{kore} to conclude. We briefly recall the argument for the reader's convenience.

\noindent As $(\gl, z, z')$ is a point of maximum, the gradient in $x$ and in $x'$ is zero and thus 
\[\nabla u(\gl x +(1-\gl)x')=\nabla u(x)=\nabla u(x').\]

\noindent As the second derivative of $C(\gl,(x+\tau,u(x+\tau)),(x'+\tau,u(x'+\tau)))$ is nonpositive in zero for every direction $\tau\in \R^m$ we get 
\[D^2u(\gl x+(1-\gl) x')-\gl D^2u(x) - (1-\gl) D^2 u(x') \le 0.\]
Using the equation satisfied by $u$, this yields the desired contradiction.
\end{proof}

 \noindent We now finally turn to the proof of our main result:
\begin{thm}\label{convTVn}
Let $g \in \Ldeufin$ be a convex function and $u$ be the minimizer of
\begin{equation*}
 \min_{ BV_\gga \cap \Ldeufin} \totvarn{u} +\frac{1}{2} \int_{\R^m} |u-g|^2 d\gga
 \end{equation*}
 then $u$ is a convex function.
 
\end{thm}

\begin{proof}
By Proposition \ref{contact} we see that if $u_R$ is the minimizer of \eqref{probepsnball} then it is convex. Arguing as in \cite[Th. 3.1]{CGN}, we see that $u_R$ converges locally uniformly to $u_\eps$ the minimizer of \eqref{probepsn}. Analogously, we can let $\eps$ goes to zero and get that $u_\eps$ converges to  $u$ the solution of \eqref{probun} which is thus convex.
\end{proof}

\noindent Let us also notice that along the same lines we can prove the following result:

\begin{thm}\label{convOUn}
Let $g$ be a convex $\Ldeufin$ function then the minimizer of 
\[\min_{u \in H^1_\gga(\R^m)} \int_{\R^m} \frac{|\nabla u|^2}{2} d\gga + \frac{1}{2}\int_{\R^m} |u-g|^2 d\gga \]
is convex.
\end{thm}

\begin{proof}
Let 
\[J_\lambda(u)=\lambda^2 \int_{\R^m} \left[\sqrt{1+\frac{|\nabla u|^2}{\lambda^2}}-1\right] d\gga + \frac{1}{2} |u-g|^2 d\gga\]
then $u_\lambda$ minimizes $J_\lambda$ if and only if it minimizes
\[ \int_{\R^m} \sqrt{\lambda^2+|\nabla u|^2} d\gga + \frac{1}{2\lambda} \int_{\R^m} |u-g|^2 d\gga. \]
 Thus $u_\lambda$ is convex. Using that for any $p$, 
\[\lim_{\lambda\to \infty}\lambda^2\left[\sqrt{1+\frac{|p|^2}{\lambda^2}}-1\right]=\frac{|p|^2}{2}\]
 we get the conclusion. 
\end{proof}

\begin{remarque}\rm
When trying to extend the previous method for more general functionals, a difficulty arise due to the lack of boundary regularity of the minimizers. 
More precisely, when reasoning as in Proposition \ref{regE},
these functionals give rise to anisotropic perimeters, for which it is not known if the minimizers of the 
corresponding obstacle problem are smooth in a neighborhood of the obstacle. 

\end{remarque}

\end{document}